\documentclass[11pt]{amsart}
\usepackage{amssymb}
\usepackage{amsmath}
\usepackage{amsmath}
\usepackage{graphicx}
\usepackage{amsfonts}
\usepackage{amssymb}
\usepackage{stmaryrd}
\usepackage{amscd}
\numberwithin{equation}{section}

\def\e{\varepsilon}

\def\t{\theta}

\def\w{\omega}

\def\R{\mathbb{R}}

\def\Z{\mathbb{Z}}

\newtheorem{thm}{Theorem}[section]
\newtheorem{lem}[thm]{Lemma}
\newtheorem{cor}[thm]{Corollary}
\newtheorem{prop}[thm]{Proposition}
\newtheorem{defin}[thm]{Definition}
\newtheorem{rem}{Remark}
\newtheorem{ex}{Example}

\begin{document}

\title[multi-dimensional-time dynamical system]
{A multi-dimensional-time dynamical system}
\author{U. A. Rozikov}
\address{Utkir Rozikov\\
Institute of Mathematics and Information Technologies\\
29, Do'rmon Yo'li str.\\
100125, Tashkent, Uzbekistan\\ email: {\tt rozikovu@yandex.ru}}
\maketitle

\begin{abstract} In this paper we give a concept of multi-dimensional-time dynamical system (MDTDS). Such dynamical system
is generated by a finite family of functions $\{f_i\}$. The multi-dimensional-time space is taken as a free group.
Using the subgroups of the free group we define periodic orbits of MDTDS and construct such orbits. We study fixed points of MDTDS and
show that the set of the fixed points is the intersection of the set of fixed points of each $f_i$. The $\omega$-limit set of the MDTDS is defined and some properties of the set is studied. Moreover we construct a MDTDS for income of a deposit from several banks and construct a MDTDS on circle. For these MDTDSs we
describe the set of fixed and periodic points. We discuss several open problems.
\end{abstract}

{\it AMS classifications (2010):} 26A18; 37A99; 37C25; 46L55


\section{Introduction}

Dynamical systems (DS) (see for example, \cite{D,E} as a discipline was born in Henri Poincar\'e's famous treatise of the three body problem. In the 1960s and 1970s a large
part of the theory of DSs concerned the case of uniformly hyperbolic DS and abstract ergodic theory of smooth DSs. However since around 1980 an emphasize has been on concrete examples of one-dimensional DSs with abundance of chaotic behavior. There are several kind of DSs, for example, real DS, discrete DS, cellular automaton etc.   (see the next section).

In all above mentioned DSs the time has dimension one.
It is known (see \cite{ge, G, Si}) that a random field is a generalization of a stochastic process (for example, a Markov chain) such that the underlying parameter need no longer be a simple real or integer valued "time," but can instead be take values that are multidimensional vectors, or points on some manifold. Since a Markov chain is a particular case of DS,
it is natural to define a multi-dimensional-time DS which generalizes notion of the random field. In this paper we define and study a multi-dimensional-time dynamical system.

In \cite{KZ} dynamics with choice has been considered, this is a generalization of discrete-time dynamics where instead of the same evolution operator at every step there is a choice of operator to transform the current state of the system.
The multi-dimensional-time DS introduced in the present paper generalizes the dynamics with choice of \cite{KZ}.  The "time" space of our DS is a free group, which allows us to define a periodicity of  "orbit" of the multi-dimensional-time DS by an analogue with periodicity of functions defined on a group.

 The paper is organized as follows. In Section 2 we give some preliminaries which will allow us to place the DS introduced in the present paper into the framework
of DSs. Section 3 contains main definitions related to multi-dimensional-time dynamical system (MDTDS). Such dynamical system
is generated by a finite family of functions $\{f_i\}$ from a subset of reals to itself. The multi-dimensional-time space is taken as a free group, there are two main reasons to choose this set as a multi-dimensional-time: using group operation we can define a shift on the time space; since in the free group there is no cycle, it is only set which can be regarded as a multi-dimensional-time space.
Using the subgroups of the free group we define periodic orbits of MDTDS and construct such orbits. Section 4 devoted to elementary properties of MDTDS.   We study fixed points of MDTDS and show that the set of the fixed points is the intersection of the set of fixed points of each $f_i$. Some properties of $\omega$-limit set of the MDTDS are studied. Moreover this section contains a result about a "Ces\'aro mean" of the MDTDS.  In Section 5  we construct a MDTDS for income of a deposit from several banks and study fixed and periodic points of such MDTDS. The last section  deals with a MDTDS on circle. For this MDTDSs we also describe the set of fixed and periodic points.

\section{Preliminaries} A {\it monoid} is an algebraic structure with a single associative binary
operation and an identity element. A monoid with invertibility property is a {\it group}.

In the most general sense, a {\it dynamical system} is a tuple $(T, M,\Phi)$ where $T$ is a monoid,
written additively, $M$ is a set and $\Phi$ is a function
$\Phi: U \subset T \times M \to M$ with
$$ I(x) = \{ t \in T : (t,x) \in U \}, \,  \Phi(0,x) = x, $$ $$  \Phi(t_2,\Phi(t_1,x)) = \Phi(t_1 + t_2, x),
\, \mbox{for} \, t_1, t_2, t_1 + t_2 \in I(x).$$

The function $\Phi(t,x)$ is called the {\it evolution function} of the dynamical system: it associates
to every point in the set $M$ a unique image, depending on the variable $t$, called the {\it evolution parameter}.
$M$ is called {\it phase space} or {\it state space}, while the variable $x$ is called {\it initial state} of the system.

Write  $\Phi_x(t) := \Phi(t,x)$, \, $\Phi^t(x) := \Phi(t,x)$ if we take one of the variables as constant.
$\Phi_x:I(x) \to M$ is called {\it flow} through $x$ and its graph trajectory through $x$. The set
$$ \gamma_x:=\{\Phi(t,x) : t \in I(x)\}$$
is called {\it orbit} through $x$.

A subset $S$ of the state space $M$ is called $\Phi$-invariant if for all $x \in S$ and all $t \in T$ one has
$\Phi(t,x) \in S$.

Varying parameters $T, M, \Phi$ one can define different kind of dynamical systems, for example, it is known the following cases:

{\it Real dynamical system:} A real dynamical system, {\it real-time dynamical system} or flow is a tuple $(T, M, \Phi)$ with $T$ an open interval in the real numbers $R$, $M$ a manifold locally diffeomorphic to a Banach space, and $\Phi$ a continuous function. If $T=R$ the system is called {\it global}, if $T$ is restricted to the non-negative reals then the system is called a {\it semi-flow}. If $\Phi$ is continuously differentiable the system is a {\it differentiable dynamical system}. If the manifold $M$ is locally diffeomorphic to $R^n$ the dynamical system is {\it finite-dimensional} and if not, the dynamical system is {\it infinite-dimensional}.

{\it Discrete dynamical system:} A discrete dynamical system, {\it discrete-time dynamical system}, map or cascade is a tuple $(T, M, \Phi)$ with $T$ the integers, $M$ a manifold locally diffeomorphic to a Banach space, and $\Phi$ a function. If $T$ is restricted to
the non-negative integers the system is called a {\it semi-cascade}.

{\it Cellular automaton:} A cellular automaton is a tuple $(T, M, \Phi)$, with $T$ the integers, $M$ a finite set, and $\Phi$ an evolution function. Some cellular automata are reversible dynamical systems, although most are not.

Recall that the {\it free group} $G\equiv G_S$ with free generating set $S$ can be constructed as follows. $S$ is a set of symbols and we suppose for every $s$ in $S$ there is a corresponding "inverse" symbol, $s^{-1}$, in a set $S^{-1}$. Let ${\bf S} = S \cup S^{-1}$, and define a word in $S$ to be any written product of elements of ${\bf S}$. That is, a word in $S$ is an element of the monoid generated by ${\bf S}$, where binary
operation for $t_1, t_2\in G$, is given by $t_1\oplus t_2=t_1t_2$ i.e., one has to write word $t_2$ after (in the right side) the word $t_1$. Thus this operation is not commutative, but it is associative.  The {\it identity} (empty word), $e$, is the word with no symbols at all.

In this paper we consider $T$ as a free group, $M$ as a set and define $\Phi$ as composition of several functions $f_i:M\to M$.
Such a dynamical system we call {\it multi-dimensional-time dynamical system} (MDTDS).

\section{Definitions and statement of the problem}

 In this paper we consider discrete dynamical system with multi-dimensional-time. The multi-dimensional-time $t$ is considered as an element of a countable free group $G=G_S$ with a finite generating set $S$ i.e., $t$ has the form $t=s_{i_1}^{\e_1}\dots s_{i_n}^{\e_n}$, with $s_i\in S$, $i\in \{i_1, \dots, i_n\}$ and $\e_i\in \Z$.  Let $X\subset \R$ be an interval.  Consider invertible functions $f_i:X\to X$, $i=1,\dots,|S|$, where $|S|$ is the number of elements of $S$.

 For $t=s_{i_1}^{\e_1}\dots s_{i_n}^{\e_n}\in G$, $x\in X$ define
 \begin{equation}\label{t}
 D_t(x)=f_{i_n}^{\e_n}(f_{i_{n-1}}^{\e_{n-1}}(\dots (f_{i_2}^{\e_2}(f_{i_1}^{\e_1}(x))\dots),
 \end{equation}
 where $f^m(x)=\underbrace{f(f(\dots f}_{m \,{\rm times}}(x)\dots).$

Consider the identity element, (which we denote by $e$) of $G$ as the {\it root} of the tree. It is easy to check that
$D_e(x)=x$ and $D_{t_2}(D_{t_1}(x)) = D_{t_1t_2}(x)$.

The discrete dynamical system with multi-dimensional-time we define as follows.
\begin{defin} The dynamical system $(T, M, \Phi)$ with $T=G$, $M=X$ and $\Phi(t,x)=D_t(x)$ is called a multi-dimensional-time dynamical system (MDTDS).
\end{defin}
\begin{defin} For an initial point $x\in X$ the set
$D^G_{f_1,\dots,f_{|S|}}(x)=\{D_t(x), t\in G\}$ is called the (full) orbit of $x$.
\end{defin}

Introduce on $G$ the structure of a graph in the following way: we call two words $x, y\in G$ {\it neighbors} (which we denote by $\langle x, y\rangle$) and connect them by an edge if there is $s\in S$ such that $x=ys$ or $y=xs$ (resp., $y=xs^{-1}$ or $x=ys^{-1}$). It is easy to see that the graph is the Cayley tree of order $2|S|-1$ (see, for example,\cite{R}).   Then any word $t=s_{i_1}^{\e_1}\dots s_{i_n}^{\e_n}$ defines a {\it path} $\pi(e,t)$ connecting $e$ with $t$ and with length (the number of edges in the shortest path) $l(\pi)=\sum_{i=1}^n|\e_i|$.

\begin{rem} 1) Recall that {\it forward orbit} of $x$ (with respect to a given function $f:X\to X$) is the set $O^+_f(x)=\{x, f(x), f^2(x),...\}=\{f^n(x), \, n\in {\mathbb N}\cup\{0\}\}$ and if $f$ is a homomorphism, then the {\it full orbit} of $x$ is the set $O_f(x)=\{f^n(x), n\in \Z\}$, and the {\it backward orbit} of $x$ is the set   $O^-_f(x)=\{x, f^{-1}(x), f^{-2}(x),...\}$.  It is easy to see that the set $D^G_{f_1,\dots,f_{|S|}}(x)$ contains all kind of orbits of an one-dimensional homomorphism, i.e., $O^{\pm}_{f_i}(x)$ and $O_{f_i}(x)\subset D^G_{f_1,\dots,f_{|S|}}(x)$, for all $i=1,\dots,|S|$ and all $x\in X$.

2) One can consider the sets $O_f^+(x)$, $O^-_f(x)$ and $O_f(x)$ as a (one-dimensional) sequences or as the set of vertices of a one-dimensional lattice, but the set  $D^G_{f_1,\dots,f_{|S|}}(x)$ has a higher dimension, i.e., it can be considered as the set of vertices of a tree with order $2|S|-1$. Thus the dynamical system constructed here is a dynamical system with multi-dimensional-time.

3) The orbit of a dynamics with choice introduced in \cite{KZ} is a subset of  the set $D^G_{f_1,\dots,f_{|S|}}(x)$ for a suitable
choice of an infinite path $\pi$.
\end{rem}

Now in terms of the group $G$ we shall give definitions of fixed point, periodic point etc, which are natural generalizations from one-dimensional-time     dynamical systems to multi-dimensional-time dynamical systems.

\begin{defin}\label{fp} The point $x$ is called a fixed point for MDTDS if $D_t(x)=x$ for any $t\in G$.
   \end{defin}
   Denote the set of fixed points by {\rm Fix}$\left(D^G_{f_1,\dots,f_{|S|}}\right)$, where $D^G_{f_1,\dots,f_{|S|}}$ is the mapping $D^G_{f_1,\dots,f_{|S|}}: x\in X\to D^G_{f_1,\dots,f_{|S|}}(x)\in X^G$ and $X^G$ denotes the set of all possible mapping from $G$ to $X$.

Let $H$ be a subgroup of $G$. For $x\in G$
denote $Hx=\{yx: y\in H\}$ the right coset of the subgroup $H$.
Define the relation $\sim$ on $G$ of right congruence by $x\sim y$
if and only if $xy^{-1}\in H$. This relation is equivalence on $G$.
Hence the right cosets of $H$ are pairwise disjoint.
The cardinal of the collection of all right cosets is called the index
of the subgroup $H$ in the group $G$ and is denoted by $|G:H|$.
Let $H\subset G$ be a subgroup with index $r\in {\mathbb N}\cup\{+\infty\}$.
Denote by $H_1=H, H_2,...,H_r$ the right cosets of $H$.

Let us describe some subgroups of $G$ (cf. \cite{GR,R1}).

For arbitrary element $u\in G$ the set
$$H_u=\{u^n: n\in \Z\}$$ is a subgroup of $G$.

 Let $n_t(s)$ be the number of occurrence of the letter
 $s\in {\bf S}$, in the reduced word $t\in G$.

 Denote $$H^{(=)}_i=\left\{t\in G: n_t(s_i)=n_t(s_i^{-1})\right\},$$ obviously this set is a subgroup of $G$.

 Let $A\subseteq {\mathcal S}=\{1,...,|S|\}$.
Then the set
$$H^{(=)}_A=\bigcap_{i\in A}H^{(=)}_i$$ is a subgroup of $G$. We denote $H^{(=)}=H^{(=)}_{\mathcal S}$.

 The set $$H_A=\left\{t\in
    G:\sum\limits_{i\in A}(n_t(s_i)+n_t(s^{-1}_i))\ \rm is\ \rm even\right\}$$ is a normal
    subgroup of index 2 of $G$.

    Some set of subgroups with index $2^m$ can be
    obtained by intersection
    $H_{A_1}\bigcap H_{A_2}\bigcap...\bigcap
    H_{A_m}$ for suitable $A_1,...,A_m\subseteq S$ (see \cite{GR}).

    Also there are normal subgroups of infinite index. Some of them
    can be described as following \cite{R2}.
    Fix $M\subseteq N_k$ such that $|M|>1$.
    Let the mapping $\pi_M:\{s_1,...,s_{|S|}\}\longrightarrow \{s_i,
    \ i\in M\}\cup \{e\}$ be defined by
    $$\pi_M(s_i)=\left\{%
\begin{array}{ll}
    s_i, & \hbox{if} \ \ i\in M \\
    e, & \hbox{if} \ \ i\notin M. \\
\end{array}
\right.$$ Denote by $G_M$ the free group generated by
$\{s_i, \ i\in M\}$. Consider $f_M:G\to G_M$ defined by
$$f_M(t)=f_M(s^{\e_1}_{i_1}s^{\e_2}_{i_2}...s^{\e_m}_{i_m})=\pi^{\e_1}_M(s_{i_1})
\pi^{\e_2}_M(s_{i_2})...\pi^{\e_m}_M(s_{i_m}).$$
    Then it is easy to see that $f_M$ is a homomorphism and hence
    $$K_M=\{t\in G: \ f_M(t)=e\}$$ is a normal subgroup of
    infinite index.

\begin{defin}\label{pp} Let $H$ be a subgroup of $G$. The point $x$ is called a $H$-periodic
point for MDTDS if $D_{rt}(x)=D_t(x)$ for any $t\in G$ and $r\in H$.
   \end{defin}

We denote the set of $H$-periodic points by Per$_H\left(D^G_{f_1,\dots,f_{|S|}}\right)$. It is easy to see that $G$-periodic points (i.e., $H=G$) are fixed points of MDTDS.

 It will be useful to introduce a partial ordering on set $G$ of multi-dimensional-time: For $t\in G$, denote  $|t|=l(\pi(e,t))$. If $t_1, t_2\in G$, we write $t_1\leq t_2$ if $t_1$ belongs to the shortest path connecting $e$ and $t_2$ and we write $t_1<t_2$ if $t_1\leq t_2$ and $t_1\ne t_2$. A sequence $t_1,t_2,t_3,\dots\in G$ is called {\it strictly increasing sequence} if $t_1<t_2<t_3<\dots$.

\begin{defin} The $\w$-limit set of $x\in X$, denoted by $\w(x)\equiv \w\left(D^G_{f_1,\dots,f_{|S|}}(x)\right)$, is the set of points $y$ for which there is an infinite strictly increasing sequence $\{t_n\}\subset G$ such that $D_{t_n}(x)\rightarrow y$ as $n\rightarrow\infty$.
   \end{defin}

   \begin{defin} Let $x$ be a $H$-periodic point. A point $y$ is called asymptotic to $x$ if there is an infinite strictly increasing sequence $\{t_n\}\subset H$ such that $D_{t_n}(y)\rightarrow x$ as $n\rightarrow\infty$. The stable set of $x$, denoted by $A(x)$ consists of all points asymptotic to $x$.
   \end{defin}

We set
\begin{equation*}
W_n=\{t\in G: |t|=n\}, \qquad V_n=\bigcup_{k=0}^n W_k.
\end{equation*}
Note that $|W_n|=q(q-1)^{n-1}$, $n\geq 1$ with $q=2|S|$ and $|V_n|=(q-2)^{-1}(q(q-1)^n-2)$. It is known that the tree generated by $G$ is a non-amenable
graph, i.e. $\inf\{{|{\rm boundary \, of}\,  W|\over |W|}: W\subset
G, 0<|W|<\infty\}>0$ for $q> 2$ (see e.g. \cite{G}). For example, one has ${|W_n|\over |V_n|}\to {q-2\over q-1}$ as $n\to\infty$.

{\it The main goal of the paper:} As in an ordinary dynamical systems our goal is to
understand the nature of all orbits  $D^G_{f_1,\dots,f_{|S|}}(x)$, $x\in X$ and to identify the set of orbits which are periodic, asymptotic, etc., i.e., our goal is to study $\w(x)$ for a given $x\in X$.

Let $M_n\subset G$ such that $M_1\subset M_2\subset M_3\subset \dots$ with $\cup_{n}M_n=G$. We also compute the limit

\begin{equation}\label{L}
\lim_{n\to \infty}C_n(x)=\lim_{n\to \infty}{\sum_{t\in M_n}D_t(x)\over |M_n|}, \ \ x\in X,
\end{equation} here $C_n(x)$ is an analogue of Ces\'aro mean $c_n={1\over n}\sum_{i=1}^na_n$ of a sequence $\{a_n\}$.

It is known that the operation of taking Ces\'aro means preserves convergent sequences and their limits. This is the basis for taking Ces\'aro means as a summability method in the theory of divergent series. There are certainly many examples for which the sequence of Ces\'aro means converges, but the original sequence does not: for example with $a_n=(-1)^{n}$ we have an oscillating sequence, but the means $c_n$ have limit $0$.

The following example shows that behavior of $C_n$ can be quite different from behavior of $c_n$.

\begin{ex} Consider $q=2|S|\geq 4$ and the "sequence" $\{A_t=(-1)^{|t|}, \, t\in G\}$ which is a "multi-dimensional-time" generalization of $a_n=(-1)^n$. Take $M_i=V_i, \, i=1,2,\dots$ then it is easy to check that
$$\sum_{t\in V_n}A_t=\left\{\begin{array}{ll}
(q-1)^n, \ \  \ {\rm if} \ \ $n$-{\rm even},\\
-(q-1)^n, \ \ {\rm if} \ \ $n$-{\rm odd}.\\
\end{array}\right.$$
Using this formula and the formula of $|V_n|$ we get
 $$\lim_{n\to\infty}{\sum_{t\in V_{2n}}A_t\over |V_{2n}|}={q-2\over q}, \ \ \lim_{n\to\infty}{\sum_{t\in V_{2n+1}}A_t\over |V_{2n+1}|}=-{q-2\over q}.$$ Therefore the limit does not converge.  But as mentioned above Ces\'aro mean $c_n$ of (one-dimensional-time) sequence $\{(-1)^n\}$ exists.
 This example shows richness the behavior of the multi-dimensional-time systems than behavior of one-dimensional-time systems.
\end{ex}
In Proposition \ref{p54} (see below) an example of the Ces\'aro mean $C_n(x)$ is presented which is convergent
with a limit $C(x)$ such that $0\leq C(x)\leq \infty$, depending on parameters the function $C(x)$ can be equal to $0$ or to a non-zero finite number or to $+\infty$.

\section{Some properties of MDTDS}

The following proposition gives a property of fixed point.

\begin{prop} The point $x$ is a fixed point of MDTDS if and only if $x$ is a fixed point of  $f_i$ for any $i=1,\dots,|S|$.
\end{prop}
\begin{proof}{\sl Necessity:} To show that $f_i(x)=x$, we take $t_1, t_2\in G$ with $t_2=t_1s_i$ then by Definition  \ref{fp}
we have $D_{t_1}(x)=D_{t_2}(x)=x$, since $D_{t_2}(x)=D_{s_i}(D_{t_1}(x))$ we get $D_{t_2}(x)=D_{s_i}(x)=f_i(x)=x$.

{\sl Sufficiency:} Straightforward.
\end{proof}
\begin{cor} $${\rm Fix}\left(D^G_{f_1,\dots,f_{|S|}}\right)=\bigcap_{i=1}^{|S|}{\rm Fix}(f_i).$$
\end{cor}
\begin{prop}\label{p2} Let $H$ be a subgroup of $G$. If point $x$ is a $H$-periodic point of MDTDS then  $x$ is a fixed point
of (sub)MDTDS $D^H_{f_1,\dots,f_{|S|}}(x)=\{D_t(x): t\in H\}$.
\end{prop}
\begin{proof} By Definition  \ref{pp}
we have $D_{yt}(x)=D_{y}(D_t(x))=D_t(x)$, for all $t\in G$ and all $y\in H$. This for $t=e$ gives $D_{y}(x)=x$ for all $y\in H$ i.e., $x\in{\rm Fix}\left(D^H_{f_1,\dots,f_{|S|}}\right)$.
\end{proof}
Thus, we have  ${\rm Per}_H\left(D^G_{f_1,\dots,f_{|S|}}\right)\subset {\rm Fix}\left(D^H_{f_1,\dots,f_{|S|}}\right)$. The following example shows that the inverse of the inclusion (i.e., $\supset$) is not true, in general.

\begin{ex} Consider $X=[0,1]$ and $f_1(x)={3\over 4}x+{1\over 4}$, and $f_2(x)=x^2$. Since $f_1$ and $f_2$ have a common fixed point, 1, the MDTDS also has fixed point 1. Take $H_{s_1s_2}=\{t\in G: t=(s_1s_2)^n=\underbrace{(s_1s_2)(s_1s_2)\dots(s_1s_2)}_{n \ \ \mbox{times}},\ \  n\in \Z\}$.
Then a point $x\in [0,1]$ is $H_{s_1s_2}$-periodic iff it satisfies the following (infinite) system of equations
$$\underbrace{(f_1f_2(f_1f_2(\dots(f_1f_2}_{n \ \ \mbox{times}}(x))\dots)=x, \ \mbox{for any} \ \ n\in \Z.$$
It is easy to see that this system has a solution $x$ iff $x$ is solution to $f_1(f_2(x))=x$. The last equation has two solutions ${1\over 3}$ and $1$.
Thus,  ${\rm Fix}\left(D^{H_{s_1s_2}}_{f_1,f_2}\right)=\{{1\over 3}, 1\}$. But we have ${\rm Per}_{H_{s_1s_2}}\left(D^G_{f_1,f_2}\right)=\{1\}$, since
for example, if we take $t=s_1^2\in G$ and $y=s_1s_2\in H_{s_1s_2}$ then $D_{yt}({1\over 3})\ne D_{t}({1\over 3})$.
\end{ex}

\begin{rem} Order the natural numbers as follows: $3\prec 5\prec 7\prec \dots \prec 2\cdot 3\prec2\cdot 5\prec 2\cdot 7\prec \dots \prec 2^2\cdot 3\prec 2^2\cdot 5\prec 2^2\cdot 7\prec \dots \prec 2^3\cdot 3\prec 2^3\cdot 5\prec 2^3\cdot 7\prec \dots \prec 2^3\prec 2^2 \prec 2 \prec 1$. Sharkovskii's theorem says that
let  $f$ be a continuous function from the reals to the reals and suppose $p\prec q$ in the above ordering. Then if $f$ has a point of least period $p$,
then $f$ also has a point of least period $q$ (see \cite{D}). If a MDTDS has $H$-periodic point and index of $H$ is equal to $n$ then we say that MDTDS has a point with period $n$. The following questions are very interesting:

How the periods of a MDTDS related with each other?

Is there a generalization of Sharkovskii's theorem for MDTDS?

How the set of periods of a MDTDS related with the set of periods of each function $f_i$, $i=1,\dots,|S|$?

These questions will be considered in a separate paper.
\end{rem}

For a fixed $t\in G$ we denote by $\omega_t(x)$ the $\w$- limit set of the one-dimensional-time DS generated by the function $D_t(x)$, i.e. $\{D_{t^n}(x), n\in \Z\}$, $x\in X$.
\begin{thm} Let $\w(x)\equiv \w\left(D^G_{f_1,\dots,f_{|S|}}(x)\right)$ be the $\w$- limit set of a MDTDS then
\begin{equation}\label{w}
\bigcup_{t\in G}\bigcup_{y\in D_{f_1,\dots,f_{|S|}}^G(x)}\w_t(y)\subset \w(x).
 \end{equation}
\end{thm}
\begin{proof} Take arbitrary $u$ from LHS of (\ref{w}) and show that $u\in \w(x)$. There is $t_0=t_0(u)\in G$ and
$y_0=y_0(u)\in D_{f_1,\dots,f_{|S|}}^G(x)$ such that $u\in \w_{t_0}(y_0)$, i.e.
there is a sequence $\{n_k\}$ such that $\lim_{k\to\infty}D_{t_0^{n_k}}(y_0)=u$. By construction $y_0$ has the following form $y_0=D_{t'}(x)$, for some $t'\in G$. For the increasing sequence $\{t_k=t't_0^{n_k}\}_{k=1,2,\dots}\subset G$, we have $$\lim_{k\to\infty}D_{t_k}(x)=\lim_{k\to\infty}D_{t_0^{n_k}}(y_0)=u,$$ i.e. $u\in \w(x)$.
\end{proof}

For $t\in G$ with $t=s_{i_1}^{\e_1}\dots s_{i_m}^{\e_m}$ we denote $\nu(t)=s_{i_m}$, i.e. the last (in the right) generator of $t$.

For $n\in \mathbb N$, $t\in V_n$, and $s\in {\bf S}$ denote
$$S_n(s)=\{s,s^2,\dots,s^n\}, \ \ V_{n,t}(s)=\{ts,ts^2,\dots, ts^{n-|t|}\}.$$
This set is generated by one generator $s$ for a fixed $t$.

The following lemma presents the set $V_n$ by subsets $S_n(s), V_{n,t}(s)$.

\begin{lem} The set $V_n$ has the following form
\begin{equation}\label{V1}
V_n=\{e\} \cup \bigcup_{s\in\bf S}S_n(s)\cup\bigcup_{k=1}^n \bigcup_{t\in W_k} \bigcup_{s\in {\bf S}\setminus \{\nu(t), (\nu(t))^{-1}\}}V_{n,t}(s).
\end{equation}
\end{lem}
\begin{proof} For any $z\in V_n$ we shall prove that $z\in$ RHS of (\ref{V1}). If $z\in \{e\}\cup \bigcup_{s\in\bf S}S_n(s)$ then nothing to prove.  Assume now
$z\in V_n\setminus (\{e\}\cup \bigcup_{s\in\bf S}S_n(s))$ then it has form $z=t(\nu(z))^\e$ with $t\in W_{|z|-|\e|}$ and $\nu(t)\ne \nu(z)$. Hence $z\in V_{n,t}(\nu(z))$, i.e.
 $z\in$ RHS of (\ref{V1}). Conversely, any $z\in V_{n,t}(s)$ satisfies $|z|\leq n$, i.e. $z\in V_n$.
\end{proof}

Assume functions $f_i$, $i=1,\dots,|S|$ satisfy the following condition: For any $x\in X$
\begin{equation}\label{f1}
{1\over n}\sum_{i=1}^nf_j^i(x)=\alpha_j+\left(a_j(x)\right)^{n}, \ \
{1\over n}\sum_{i=1}^nf_j^{-i}(x)=\beta_j+\left(b_j(x)\right)^{n}, \ \ j=1,\dots,|S|,
\end{equation}
where $\alpha_j$, $\beta_j$ do not dependent on $x$ and functions $a_j(x)$, $b_j(x)$ such that $0<a_j(x)\leq a_j$, $0<b_j(x)\leq b_j$, $\forall x\in X$ with
\begin{equation}\label{o}
a_j<2|S|-1, \ \ b_j<2|S|-1, \ \forall j=1,\dots,|S|.\end{equation}

Denote $$A=\sum_{j=1}^{|S|}(\alpha_j+\beta_j).$$

Note that computation of the limit (\ref{L}) is difficult, in many cases the limit may be infinite. The following theorem gives conditions on functions $f_j$ under which the limit is bounded.

\begin{thm} If $|S|>1$ and the condition (\ref{f1}) is satisfied then for any $x\in X$ the following holds
\begin{equation}\label{E} {(q-1)A\over q(q-2)}\leq  \lim_{n\to\infty}{\sum_{t\in V_n}D_t(x)\over |V_n|}\leq {(q-1)A\over q(q-2)}+(q-2)\sum_{j=1}^{|S|}
\left({a_j\over (q-a_j-1)^2}+{b_j\over (q-b_j-1)^2}\right),\end{equation}
where $q=2|S|$.
\end{thm}
\begin{proof} Using formulas (\ref{t}) and (\ref{V1}) we get
\begin{equation}\label{s}
\sum_{t\in V_n}D_t(x)=x+\sum_{j=1}^{|S|}\sum_{i=1}^n\left(f_j^i(x)+f_j^{-i}(x)\right)+\Xi_n,
\end{equation}
where
\begin{equation}\label{U}\Xi_n=\sum_{k=1}^n\sum_{t\in W_k:\atop \nu(t)\in S}\sum_{j=1\atop s_j\ne \nu(t)}^{|S|}\sum_{i=1}^{n-k}\left(f_j^i(D_t(x))+f_j^{-i}(D_t(x))\right)+\end{equation}
$$
\sum_{k=1}^n\sum_{t\in W_k:\atop \nu(t)\in S^{-1}}\sum_{j=1\atop s_j\ne (\nu(t))^{-1}}^{|S|}\sum_{i=1}^{n-k}\left(f_j^i(D_t(x))+f_j^{-i}(D_t(x))\right).$$
Now using the conditions of the theorem we get
\begin{equation}\label{X}\Xi_n\leq 2\sum_{k=1}^n(n-k)\sum_{t\in W_k: \atop \nu(t)\in S}\left(A-\alpha_{j_t}-\beta_{j_t}+\sum_{j=1\atop s_j\ne \nu(t)}^{|S|}\left(a^{n-k}_j+b^{n-k}_j\right)\right),
\end{equation}
where $j_t$ is unique index with $s_{j_t}=\nu(t)$.

It is easy to see that $|\{t\in W_k:\nu(t)=s_i\}|=(q-1)^{k-1}$, consequently we have the following
$$\sum_{t\in W_k: \atop \nu(t)\in S}(\alpha_{j_t}+\beta_{j_t})=\sum_{i=1}^{|S|}\sum_{t\in W_k:\atop \nu(t)=s_i}(\alpha_i+\beta_i)=\sum_{i=1}^{|S|}|\{t\in W_k:\nu(t)=s_i\}|(\alpha_i+\beta_i)=A(q-1)^{k-1}.$$
Using this formula, positivity of $a_j$, $b_j$ and formula of $|W_k|$ from (\ref{X}) we get
\begin{equation}\label{X1}\Xi_n\leq \sum_{k=1}^n(n-k)(q-1)^{k-1}\left((q-1)A+q\sum_{j=1}^{|S|}\left(a^{n-k}_j+b^{n-k}_j\right)\right).
\end{equation}
The following formula is known
\begin{equation}\label{n}
\sum_{k=1}^nkx^k=x{1-x^n\over (1-x)^2}-{nx^{n+1}\over 1-x}.
\end{equation}
By (\ref{n}) from (\ref{X1}) we obtain
 \begin{equation}\label{X2}\Xi_n\leq {q-1\over q-2}A\left({(q-1)^n-1\over q-2}-n\right) +q\sum_{j=1}^{|S|}\left({a_j((q-1)^n-a_j^n)\over (q-a_j-1)^2}-{na^{n}_j\over q-a_j-1}\right)+
\end{equation}
$$q\sum_{j=1}^{|S|}\left({b_j((q-1)^n-b_j^n)\over (q-b_j-1)^2}-{nb^{n}_j\over q-b_j-1}\right).$$
Dividing both side of (\ref{s}) to $|V_n|$ then takeing limit at $n\to\infty$, by (\ref{o}) and (\ref{X2}) we get RHS of (\ref{E}).

To get LHS of (\ref{E}) one can use the following estimation
\begin{equation}\label{X3}
\Xi_n> (q-1)A\sum_{k=1}^n(n-k)(q-1)^{k-1}=(q-1)A\left({(q-1)^n-1\over (q-2)^2}-{n\over q-2}\right),
\end{equation}
which follows from (\ref{U}) for $a_j(D_t(x))=b_j(D_t(x))=0$, $j=1,\dots,|S|$.

\end{proof}
\section{MDTDS for income of a deposit from several banks}

In this section we consider a concrete example of MDTDS for which we give exact calculations of the set of periodic points.
Assume that there are $|S|$ banks, the
income from the bank $i$ for a deposit is fixed at $p_i>0$, ($i=1,\dots,|S|$) percent.
If the deposit was $x$ then after one unite of time it becomes $q_ix$, where $q_i=1+{p_i\over 100}$.
Consider the MDTDS with $X=(0,+\infty)$, and functions $f_i(x)=q_ix$, $i=1,\dots,|S|$.
In this case for $t=s_{i_1}^{\e_1}\dots s_{i_n}^{\e_n}\in G$, $x\in X$ we have
 \begin{equation}\label{b}
 D_t(x)\equiv D_{q_1,\dots,q_{|S|}}^G(x)=q_{i_n}^{\e_n}q_{i_{n-1}}^{\e_{n-1}}\dots q_{i_2}^{\e_2}q_{i_1}^{\e_1}x.
 \end{equation}
\begin{prop}\label{b1}
Let $H$ be a subgroup of $G$. A point $x\in (0,+\infty)$ is $H$-periodic (for MDTDS (\ref{b})) if and only if
\begin{equation}\label{5}
q_{j_1}^{\delta_1}q_{j_{2}}^{\delta_{2}}\dots q_{j_m}^{\delta_m}=1
\end{equation}
for any $y=s_{j_1}^{\delta_1}\dots s_{j_m}^{\delta_m}\in H$.
\end{prop}
\begin{proof}
By the definition we have $x$ is $H$ periodic iff
$$q_{j_1}^{\delta_1}q_{j_{2}}^{\delta_{2}}\dots q_{j_m}^{\delta_m}q_{i_1}^{\e_1}q_{i_{2}}^{\e_{2}}\dots q_{i_n}^{\e_n}x=q_{i_1}^{\e_1}q_{i_{2}}^{\e_{2}}\dots q_{i_n}^{\e_n}x,$$
for any $t=s_{i_1}^{\e_1}\dots s_{i_n}^{\e_n}\in G$ and any $y=s_{j_1}^{\delta_1}\dots s_{j_m}^{\delta_m}\in H$.
This gives (\ref{5}).\end{proof}
For a subgroup $H$ of $G$ denote by $I(H)=H\cap S$ the set of generators belonging in $H$.
\begin{prop}\label{b2}
If $I(H)\ne \emptyset$ then
$${\rm Per}_H(D_{q_1,\dots,q_{|S|}}^G)=\emptyset.$$
\end{prop}
\begin{proof} Assume $s_{i_0}\in H$ then from equation (\ref{5}) for $y=s_{i_0}$ we get $q_{i_0}=1$, but we have condition
$q_i>1$ for all $i$.\end{proof}

Note that any subgroup ${\tilde H}$ of $H^{(=)}$ or $H_{\mathcal S}$ can be an
example of a subgroup with $I({\tilde H})=\emptyset$.
\begin{prop}\label{b3} For any subgroup $H$ of $H^{(=)}$ we have
$${\rm Per}_H(D_{q_1,\dots,q_{|S|}}^G)=(0,+\infty).$$
\end{prop}
\begin{proof} We have $n_t(s_i)=n_t(s^{-1}_i)$ for any $t\in H$ and any $i=1,\dots,|S|$. Thus in LHS of (\ref{5}) we have
$$q_{j_1}^{n_t(s_{j_1})-n_t(s^{-1}_{j_1})}\dots q_{j_m}^{n_t(s_{j_m})-n_t(s^{-1}_{j_m})}=(q_{j_1}q_{j_{2}}\dots q_{j_m})^0=1.$$
\end{proof}
But the following example shows that the condition $I(H)=\emptyset$ is not sufficient to existence of a $H$-periodic point.

\begin{ex} Consider $t\in H_{\mathcal S}$ and $t'=ts_1^2\in H_{\mathcal S}$. If we write the condition (\ref{5}) for $t$ and $t'$ then
  these equalities give that $q_1^2=1$ which is impossible. Hence $${\rm Per}_{H_{\mathcal S}}(D_{q_1,\dots,q_{|S|}}^G)=\emptyset.$$
  \end{ex}
 \begin{prop}\label{p54} If $M_n=V_n$, $|S|>1$ then for any $x\in (0,+\infty)$ we have
 \begin{equation}\label{V}
 \lim_{n\to \infty}{\sum_{t\in V_n}D_t(x)\over |V_n|}=\left\{\begin{array}{ll}
 0, \ \ \mbox{if} \ \ Q<q-1;\\[3mm]
 {(q-2)x\over q\prod_{i=1}^{|S|}(1-q_i^{-1})}, \ \ \mbox{if} \ \ Q=q-1;\\[3mm]
 +\infty, \ \ \mbox{if} \ \ Q>q-1,\\[3mm]
 \end{array}\right.
 \end{equation}
 where $Q=\prod_{i=1}^{|S|}q_i$ and $q=2|S|$.
 \end{prop}
  \begin{proof} Let $\Z_n=\{-n,\dots,n\}$.
  We have
  $$\sum_{t\in V_n}D_t(x)=x\sum_{t\in V_n}q_{i_1}^{\e_{i_1}}\dots q_{i_m}^{\e_{i_m}}=
  x\sum_{\e_1\in \Z_n}\dots\sum_{\e_{|S|}\in \Z_n}q_1^{\e_1}\dots q_{|S|}^{\e_{|S|}}=$$
  $$x\prod_{i=1}^{|S|}\sum_{\e\in \Z_n}q_i^{\e}=x\prod_{i=1}^{|S|}{q_i^{n+1}-q_i^{-n}\over q_i-1}.$$
  Using this equality and formula of $|V_n|$ we get
  $$\lim_{n\to \infty}{\sum_{t\in V_n}D_t(x)\over |V_n|}=x{q-2\over \prod_{i=1}^{|S|}(1-q_i^{-1})}
\lim_{n\to \infty}{(Q/(q-1))^n\over q-2(q-1)^{-n}}\prod_{i=1}^{|S|}\left(1-q_i^{-2n-1}\right).$$
Which gives the equality (\ref{V}).\end{proof}

\section{MDTDS on circle} In this section we consider the following example of MDTDS: $X=[0,1]$,  mappings $f_i : [0,1] \rightarrow [0,1]$ are
given by
$$f_i(x) = x + \theta_i \mod 1, \ \ i=1,\dots, |S|,$$
where $\theta_i$ are given positive numbers. These mappings rotations by an angle of $\theta_i\times 360$ degrees on a circle
after identifying that circle with the interval $[0, 1]$ where the boundary points are identified (that is $\R/\Z$).

It is known that such a rotation is an element of infinite order in the circle group. If $\theta_i$ were rational, then the rotation would be an element of finite order. In other words, if $\theta_i$ were rational, then applying the rotation a sufficient number of times would map all elements of the circle back on to themselves.
If $\theta_i$ is an irrational number, then for any initial point this will generate a dense set in the interval $[0, 1)$  by repeatedly applying the mapping $f_i$ to it as an iterated function. In other words for any $x$ the set  $\{x+n \theta_i : n \in \Z\}$
is dense in the circle.

Using the property $((x+y)\mod 1+z)\mod 1=(x+y+z)\mod 1$ we obtain
\begin{equation}\label{c}
D_t(x)=f_{i_n}^{\e_n}(f_{i_{n-1}}^{\e_{n-1}}(\dots (f_{i_2}^{\e_2}(f_{i_1}^{\e_1}(x))\dots)=(x+\e_1\theta_{i_1}+\dots+\e_n\theta_{i_n})\mod 1.
\end{equation}
So it is convenient to denote $D^G_{\t_1,\dots,\t_{|S|}}=D^G_{f_1,\dots,f_{|S|}}$.

The formula (\ref{c}) gives the commutativity $D_{ty}(x)=D_{yt}(x)$ for any $x\in [0,1]$ and $t,y\in G$.

In this case we get
\begin{equation}\label{f}
{\rm Per}_H\left(D^G_{\t_1,\dots,\t_{|S|}}\right)={\rm Fix}\left(D^H_{\t_1,\dots,\t_{|S|}}\right),
\end{equation}
 which is
not true in general (see Proposition \ref{p2}).

\begin{thm}\label{t1} The set of fixed points is
$${\rm Fix}\left(D^G_{\t_1,\dots,\t_{|S|}}\right)=\left\{\begin{array}{ll}
\emptyset, \ \ \mbox{if}\ \ \theta_{i_0} \ \ \mbox{is not integer, for some}\ \  i_0\in \{1,...,|S|\};\\[3mm]
[0,1), \ \ \mbox{if}\ \ \t_i\in \Z \ \ \mbox{for all}\ \  i=1,\dots,|S|.\\
\end{array}\right.$$
\end{thm}
\begin{proof} For a fixed point $x\in [0,1]$ we must have
$$(x+\e_1\theta_{i_1}+\dots+\e_n\theta_{i_n})\mod 1=x,\ \ \mbox{for any}\ \ t=s^{\e_1}_{i_1}\dots s^{\e_n}_{i_n}\in G.$$
This gives
\begin{equation}\label{3}
\e_1\theta_{i_1}+\dots+\e_n\theta_{i_n}=[x+\e_1\theta_{i_1}+\dots+\e_n\theta_{i_n}],\ \ \mbox{for any}\ \ t=s^{\e_1}_{i_1}\dots s^{\e_n}_{i_n}\in G,
\end{equation}
 where $[a]$ denotes integer part of $a$.
Assume $\t_{i_0}\notin \Z$ then we take $t=s_{i_0}$ and for this $t$ the equation (\ref{3}) becomes $\t_{i_0}=[x+\t_{i_0}]$ which
has not solution. Now if all $\t_i$ are integer numbers then we have that $\e_1\theta_{i_1}+\dots+\e_n\theta_{i_n}$ is also an integer number. Thus any $x\in [0,1)$ satisfies the equation (\ref{3}).
\end{proof}

But description of the set ${\rm Per}_H\left(D^G_{\t_1,\dots,\t_{|S|}}\right)$ is difficult, the difficulty depends on the given subgroup $H$.

    For given $t=s^{\e_1}_{i_1}\dots s^{\e_m}_{i_m}\in G$ denote  $q(t)=\e_1\t_{i_1}+\dots+\e_m\t_{i_m}$.
\begin{thm}\label{t2} For any subgroup $H$ of $G$, we have
$${\rm Per}_H\left(D^G_{\t_1,\dots,\t_{|S|}}\right)=\left\{\begin{array}{ll}
\emptyset, \ \ \mbox{if}\ \ q(t_0)\notin \Z,\ \  \mbox{for some} \ \ t_0\in H\\[3mm]
[0,1), \ \ \mbox{if}\ \ q(t)\in \Z, \ \ \mbox{for any} \ \ t\in H.\\
\end{array}\right.$$
\end{thm}
\begin{proof} From formula (\ref{f}) it follows that $x\in [0,1]$ is $H$-periodic iff
$$(x+q(t))\mod 1=x,\ \ \mbox{for any}\ \ t\in H.$$
Consequently
\begin{equation}\label{4}
q(t)=[x+q(t)],\ \ \mbox{for any}\ \ t\in H.
\end{equation}

It is easy to see that equation (\ref{4}) has no solution if $q(t)\notin \Z$, for some $t=t_0\in H$
and any $x\in [0,1)$ is a solution if $q(t)\in \Z$ for any $t\in H$.
\end{proof}

\begin{rem}
 From Proposition \ref{b1} and Theorem \ref{t2} the following interesting problems arise:

a)  For given $q_i>1$, $i=1,\dots, |S|$, find a subgroup $H$ of $G$ such that
 $q_{j_1}^{\delta_1}q_{j_{2}}^{\delta_{2}}\dots q_{j_m}^{\delta_m}=1$
for any $y=s_{j_1}^{\delta_1}\dots s_{j_m}^{\delta_m}\in H$;

 b) Find a subgroup $H$ of $G$ with $q(t)\in \Z$ for all $t\in H$.\end{rem}

 Denote
 $${\mathcal H}\equiv {\mathcal H}_{\t_1,\dots,\t_{|S|}}=\{H\subset G: H \, \mbox{is a subgroup with} \, q(t)\in \Z, \, \forall t\in H\}.$$

 For class $\{H_u: u\in G\}$ of subgroups $H_u$ constructed above we have
 \begin{prop}
 $H_u\in {\mathcal H}$ iff $q(u)\in \Z.$
 \end{prop}
\begin{proof} Since any $t\in H_u$  has form $t=u^n$, $n\in \Z$ we get  $q(t)=nq(u)$. Consequently $q(t)\in \Z$ iff $q(u)\in \Z$.
\end{proof}

\begin{thm}\label{t3} If $\t_{i_0}$ is a rational number for some $i_0\in \{1,\dots,|S|\}$ then there is a subgroup $H_0$ of $G$  such
  that   $${\rm Per}_{H_0}\left(D^G_{\t_1,\dots,\t_{|S|}}\right)=[0,1)$$
\end{thm}
\begin{proof} Let $\t_{i_0}={p\over m}$ be a rational number.  Take $u=s^m_{i_0}$ then for any $t\in H_0=H_{u}=\{s_{i_0}^{mn}: n\in \Z\}$ we have
$q(t)=mn\t_{i_0}=np\in \Z$ then  Theorem \ref{t2} completes the proof.
\end{proof}

\begin{rem}
 For (ordinary) one-dimensional-time dynamical systems it is clear that if a point is
 periodic with prime period $p$ then it is also periodic with (non prime) period $rp$, $r\in \mathbb N$, but it has not any other kind of periodicity.
 But in MDTDS the structure is rich: for example, the analogue of the $rp$ periodicity in MDTDS for a $H_0$-periodic point will
 $H_{r}$-periodicity with $H_{r}=\{s_{i_0}^{rmn}: n\in \Z\}$, but the same point may have other periodicity, for instance if $\t_{i_1}={p_1\over m_1}$ is rational with  $i_1\ne i_0$ then ${\rm Per}_{H'}\left(D^G_{\t_1,\dots,\t_{|S|}}\right)=[0,1)$ with $H'=\{s_{i_1}^{m_1n}: n\in \Z\}$. So in this case any point $x\in [0,1)$ has $H_0$-periodicity and $H'$-periodicity with $H'\ne H_r$ i.e., there is not such periodicity in one-dimensional-time dynamical systems. In general (case of circle mappings), Theorem \ref{t2} gives that if a point $x$ is $H_0$-periodic for some subgroup
 then it is $H$-periodic for any $H\in {\mathcal H}$.
 \end{rem}

\section*{Acknowledgements}  This work done under Junior Associate scheme in the Abdus Salam International
Center for Theoretical Physics (ICTP), Trieste, Italy and I thank ICTP for providing
financial support and all facilities.
I also thank Institut des Hautes \'Etudes Scientifiques (IHES), Bures-sur-Yvette, France for support of my visit to IHES (January-April 2012) and IMU/CDC-program for travel support.


\vskip 0.3 truecm

\begin{thebibliography}{99}

\bibitem{D} R. L. Devaney, An Introduction to Chaotic Dynamical System (Westview Press, 2003).

\bibitem{E} S. N. Elaydi, Discrete Chaos (Chapman Hall/CRC, 2000).

\bibitem{GR}  N.N. Ganikhodjaev, U.A. Rozikov, Description of non periodic extreme Gibbs measures of some
models on Cayley tree. {\it Theor. Math. Phys.} {\bf 111}: 109-117 (1997).

\bibitem{ge} H.-O. Georgii, \emph{Gibbs measures and phase transitions}, de
Gruyter Studies in Mathematics, 9. Walter de Gruyter  \& Co.,
Berlin, (1988).

\bibitem{G} G. Grimmett,  The random-cluster model (Berlin: Springer, 2006).

\bibitem{KZ} L. Kapitanski, S. \~Zivanovi\'c,  Dynamics with choice. {\it Nonlinearity} {\bf 22}: 163--186  (2009).

\bibitem{R} U.A. Rozikov, Representation of trees and its applications. {\it Math. Notes.} {\bf 72}: 516--527 (2002).

\bibitem{R1} U.A.  Rozikov, Partition structures of the Cayley  tree and applications for describing periodic
Gibbs distributions. {\it Theor. Math. Phys.} {\bf 112}: 929--933 (1997).

\bibitem{R2} U.A. Rozikov, Countably periodic Gibbs measures of the Ising model on the Cayley tree. {\it
Theor. Math. Phys.} {\bf 130}: 92-100 (2002).

\bibitem{Si} Sinai Ya.G., \textit{Theory of phase transitions: rigorous
results}, International Series in Natural Philosophy, 108. Pergamon
Press, Oxford-Elmsford, N. Y., 1982.

\end{thebibliography}
\end{document}